\documentclass[12pt]{article}
\usepackage{amsthm,amsmath,amssymb}
\usepackage{graphics}

\newtheorem{la}{Lemma}[section]
\newtheorem{pro}[la]{Proposition}
\newtheorem{thm}[la]{Theorem}
\newtheorem{cor}[la]{Corollary}

\newtheorem{example}{Example}

\newcommand{\bex}{\begin{example}}
\newcommand{\eex}{\end{example}}
\date{}
\title{\bf Beyond Hamiltonicity of Prime Difference Graphs}

\author{Hong-Bin Chen\thanks{Department of Applied Mathematics, National Chung Hsing University, Taichung 40249, Taiwan
{\tt (Email: andanchen@gmail.com)} Supported by MOST 107-2115-M-035-003-MY2}
\and
Hung-Lin Fu\thanks{Department of Applied Mathematics, National Chiao Tung University, Hsinchu 30050, Taiwan
{\tt (Email: hlfu@math.nctu.edu.tw)} Research supported by MOST 108-2115-M-009-010}
\and
Jun-Yi Guo\thanks{Department of Mathematics, National Taiwan Normal University, Taipei 11677, Taiwan
{\tt (Email: davidguo@ntnu.edu.tw)} Research supported by MOST 108-2115-M-003-014}
}

\begin{document}

\maketitle

\begin{abstract}
A graph is Hamiltonian if it contains a cycle which visits every vertex of the graph
exactly once. In this paper, we consider the problem of Hamiltonicity of a graph $G_n$, which will be called the prime difference graph of order $n$, with vertex set $\{1,2,\cdots, n\}$ and edge set $\{uv: |u-v|$ is a prime number$\}$.
A recent result, conjectured by Sun and later proved by Chen, asserts that $G_n$ is Hamiltonian for $n\geq 5$. This paper extends their result in three directions. First, we prove that for any two integers $a$ and $b$ with $1\leq a<b\leq n$, there is a Hamilton path in $G_n$ from $a$ to $b$ except some cases of small $n$. This result implies robustness of the Hamiltonicity property of the prime difference graph in a sense that for any edge $e$ in $G_n$ there exists a Hamilton cycle containing $e$. Second, we show that the prime difference graph contains considerably more about the cycle structure than Hamiltonicity; precisely, for any integer $n\geq 7$, the prime difference graph $G_n$ contains any 2-factor of the complete graph of order $n$ as a subgraph. Finally, we find that $G_n$ may contain more edge-disjoint Hamilton cycles. In particular, these Hamilton cycles are generated by two prime differences.\\

 \noindent \textbf{Keywords:} Hamilton cycle; Prime difference graph; Prime sum graph

 \noindent \textbf{Mathematics Subject Classifications:} 05C45
\end{abstract}

\section{Introduction}

A {\it Hamilton path} ({\it cycle}) in a graph $G$ is a path (cycle) visiting each vertex of $G$ exactly once.  A graph is {\it Hamiltonian}
if it contains a Hamilton cycle. The notion of Hamiltonicity is a very important and extensively studied concept in graph theory. Determining whether such paths and cycles exist in graphs is an extremely difficult problem, which is NP-complete. Many efforts have been devoted to obtain sufficient conditions for Hamiltonicity. One of the oldest results is the theorem of Dirac \cite{dirac}, who showed that if the minimal degree of graph $G$ on $n$ vertices is at least $n/2$, then $G$ contains a Hamilton cycle. This result is only one example of a vast majority of known sufficient conditions for Hamiltonicity that mainly deal with graphs which are fairly dense. It appears that little is known about Hamiltonicity in relatively sparse graphs. Recently, there has been an increasing interest in studying the {\it robustness of Hamiltonicity}, which is aiming to measure ``how strongly a graph possesses the property of being Hamiltonian?''. Several different measures of the robustness of Hamiltonicity have been concerned; for example, a graph has many Hamilton cycles or many edge-disjoint Hamilton cycles, a graph has a Hamilton cycle passing through any designated edge, and how many edges one must delete to destroy its Hamiltonicity. See the surveys \cite{gould,sudakov} for more closely related topics.

In 1982, Filz \cite{filz} first introduced the notion ``prime circle'' of positive integers, which makes an unusual connection between  numbers and Hamilton cycles. A {\it prime circle of order $2n$} is a cyclic permutation of the integers 1 through $2n$ such that the sum of any two adjacent integers is a prime number. For instance, $\langle1,4,3,2,5,6\rangle$ is a prime circle of order 6. Filz asked the question that ``Are there prime circles of order $2n$ for any positive integer $n$?''. It is a problem concerning numbers essentially and can be found also in the famous book entitled ``Unsolved Open Problems in Number Theory'' by Richard Guy \cite{guy}. However, this question has been attracted little attention and no progress has been made until a recent paper by Chen, Fu, and Guo \cite{cfg}. They studied the problem by defining a graph, called the {\it prime sum graph} $SG_{n}$ of order $n$, associated vertices with the integers $\{1, 2, \cdots, n\}$ and connecting two vertices $i$ and $j$ if and only if $i+j$ is a prime number. It was proved by Greenfield and Greenfield \cite{greenfield} and reproduced by Galvin \cite{Galvin} that $SG_{n}$ contains a perfect matching if and only if $n$ is even. In addition, Filz's conjecture can be rephrased as that $SG_{2n}$ contains a Hamilton cycle for any integer $n\geq 2$. Recently, by adapting the result of the Bounded Gap Conjecture \cite{maynard, polymath}, Chen et al. have proved that there exist infinitely many integers $n$ such that $SG_{2n}$ contains a Hamilton cycle.

In a lecture in 2012, Sun conjectured these prime circles can also be obtained when replacing the prime sum with the prime difference. Precisely, he conjectured that for any positive integer $n\geq 5$, there exists a cyclic permutation $a_1,a_2,\cdots,a_n$ of the integers $1,2, \cdots, n$ such that $|a_1-a_{2}|, |a_2-a_3|, \cdots, |a_{n-1}-a_n|, |a_n-a_1|$ are prime. Soon after, Chen and Sun \cite{oeis} verified this conjecture by providing a simple construction. Similarly, one can define a graph $G_n$ whose vertex set is $\{1, 2, \cdots, n\}$ and two vertices $i$ and $j$ are adjacent if and only if $|i-j|$ is a prime. We call $G_n$ the {\it prime difference graph of order $n$}.

From the point of view of graphs, Chen and Sun's result can be further extended. Their result can be rephrased as the prime difference graph $G_n$ has a Hamilton cycle for $n\geq 5$. In fact, it is easily shown from their result that the prime difference graph $G_n$, $n\geq 6$, is {\it pancyclic}, which means there are cycles of each length $\ell$ for $3\leq \ell \leq n$. Since $G_{n-1}\subseteq G_{n}$, $G_n$, $n\geq 6$, contains cycles of each length $\ell$,  $5\leq \ell \leq n$. Together with the observation that $G_6$ contains a 3-cycle $(136)$ and a 4-cycle $(1463)$, one can conclude that $G_n$, $n\geq 6$, is pancyclic.

In this paper, we are interested in studying more about the robustness of Hamiltonicity and the cycle structure in the prime difference graphs. Our main results show how to strengthen Chen and Sun's result in three directions. The rest of this paper is organized accordingly.
In Section 2, we prove that for any two integers $n_1$ and $n_2$, $1\leq n_1 < n_2\leq n$, $G_n$ contains a Hamilton path from $n_1$ to $n_2$ provided $n\geq 5$ except for some particular cases of small $n$. This implies that for any designated edge $e$ in $G_n$, $n\geq 5$, there exists a Hamilton cycle passing through $e$. In Section 3, we show that, for $n\geq 7$, $G_n$ contains any 2-factor of the complete graph of order $n$ as a subgraph. Finally, in Section $4$, we study  Hamilton cycles in $G_n$ which are generated by two prime differences $p$ and $q$, that is, all edges are of differences either $p$ and $q$.

\section{Hamilton paths with any designated end-vertices}

This section mainly focuses on constructing a Hamilton path with any designated end-vertices in the prime difference graph. As a consequence, this implies that for any edge $e$ in the prime difference graph there exists a Hamilton cycle containing $e$. As our strategy is based on a recursive construction, we need some basic notations and propositions before going through the details of its proof.

We first extend the notion of the prime difference graphs to a more general setting. Let $[m,n]$ denote the set of positive integers in the interval of $m$ and $n$, i.e., $\{m, m+1, \cdots, n\}$. For any interval of consecutive integers $[m,n]$, define $G_{m,n}$ as the prime difference graph associated vertices with $[m,n]$. Obviously, $G_{1,n}$ is indeed $G_n$.
We use $\langle a_1,a_2,\cdots,a_\ell\rangle$ as a path notation  to record the ordering in the path. Sometimes, for clarity, we use $n_1\xrightarrow{[m,n]} n_2$ to omit the ordering but only to record two endpoints of a path from $n_1$ to $n_2$ which uses all of the integers in the interval $[m,n]$.

The following propositions will be repeatedly used in the proofs of the main results.

\begin{pro}\label{sym} We have the following two properties:\\
(1) {\bf Complement property}: In $G_n$, $\langle v_1, v_2, \dots, v_{n}\rangle$ is a Hamilton path if and only if $\langle n+1-v_1, n+1-v_2, \dots, n+1-v_{n}\rangle$ is a Hamilton path.\\
(2) {\bf Shift property}: For any positive integers $a, b$ and $k$, $G_{a,b}$ is isomorphic to $G_{a+k,b+k}$. Moreover, $\langle v_1, v_2, \dots, v_{b-a+1}\rangle$ is a Hamilton path of $G_{a,b}$ if and only if $\langle v_1 +k, v_2 +k, \dots ,v_{b-a+1}+k\rangle$ is a Hamilton path of $G_{a+k,b+k}$.
\end{pro}
\begin{proof} Both properties follow directly from a simple observation. \end{proof}

Given $n_1, n_2 \in [m,n]$, we say a graph $G_{m,n}$ is {\it $(n_1,n_2)$-feasible} if there exists a Hamilton path from $n_1$ to $n_2$ in $G_{m,n}$. For example, $G_7$ is $(1,3)$-feasible since $\langle 1,6,4,2,7,5,3\rangle$ is a Hamilton path in $G_7$.

\begin{la}\label{la21}
For each $n\geq 6$, $G_n$ is $(1,m)$-feasible where $2\leq m \leq 6$.
\end{la}
\begin{proof}
For the sake of clarity, we split the proof into five cases.\\
{Case 1:} $m =2$\\
By direct construction, $\langle 1,4,6,3,5,2\rangle$ and $\langle 1,4,6,3,5,7,2\rangle$ are Hamilton paths of $G_6$ and $G_7$ from $1$ to $2$ respectively.
Now we will construct a Hamilton path from $1$ to $2$ in $G_n$ for each $n\geq 8$ recursively. Assume that we have constructed  a Hamilton path from $1$ to $2$ in $G_{n-2}$. Then, by the shift property, $G_{3,n}$ has a Hamilton path $P$ from $3$ to $4$. Now, $\langle 1,P,2\rangle$ is a Hamilton path in $G_n$ as desired.\\
{Case 2:}  $m =3$\\
By direct construction, we have the following Hamilton paths of $G_n$, $n=6,7,8,9$, respectively: $\langle 1,6,4,2,5,3\rangle$, $\langle 1,6,4,2,7,5,3\rangle$, $\langle 1,4,2,7,5,8,6,3\rangle$ and $\langle 1,4,2,5,7,9,6,8,3\rangle$. Again, we construct the desired Hamilton path recursively for $n\geq 10$.
Assume that such a path $P'$ has been constructed for $G_{n-4}$. For $n\geq 10$, $G_{5,n}$ has a Hamilton path $P$ from $5$ to $6$ following from the construction of $P'$ and the shift property. Now, it implies that $\langle 1,4,2,P,3\rangle$ is a Hamilton path we need.

Since the construction can be obtained by a similar way, in the following we only list the initial cases and the corresponding construction. \\
{Case 3:}  $m =4$\\
(1) $\langle 1,3,5,2,4\rangle$, $\langle 1,6,3,5,2,4\rangle$, $\langle 1,6,3,5,2,7,4\rangle$, $\langle 1,3,6,8,5,7,2,4\rangle$ are Hamilton paths from $1$ to $4$ for $5\leq n\leq 8$, respectively.\\
(2) Let $P$ be a Hamilton path from $5$ to $7$ defined on $G_{5,n}$ for $n\geq 9$. Then, $\langle 1,3,P,2,4\rangle$ is a Hamilton path.\\
{Case 4:}  $m =5$\\
(1) Initial cases are $\langle 1,3,6,4,2,5\rangle$, $\langle 1,3,6,4,2,7,5\rangle$, $\langle 1,8,3,6,4,2,7,5\rangle$, $\langle 1,4,2,7,9,6,3,8,5\rangle$ and $\langle 1,4,2,9,6,3,8,10,7,5\rangle$ for $6\leq n\leq 10$, respectively.\\
(2) For $n\geq 11$, let $P$ be a Hamilton path from $6$ to $9$ defined on $G_{6,n}$. Then $\langle 1,3,P,4,2,5\rangle$ is a Hamilton path.\\
{Case 5:}  $m =6$\\
(1) Initial cases are $\langle 1,3,5,2,4,6\rangle$, $\langle 1,4,7,2,5,3,6\rangle$, $\langle 1,8,3,5,7,2,4,6\rangle$, $\langle 1,4,7,9,2,5,3,8,6\rangle$ and $\langle 1,4,2,5,3,8,10,7,9,6\rangle$ for $6\leq n\leq 10$, respectively.\\
(2) Let $P'$ be a Hamilton path from $6$ to $7$ defined on $G_{6,n}$ for $n\geq 11$ and $P$ be a Hamilton path from $7$ to $6$ by reversing the order of $P'$. Then, $\langle 1,3,5,2,4,P\rangle$ is the desired Hamilton path of $G_n$ from $1$ to $6$.
\end{proof}

With the above lemma, we can further extend it.
\begin{pro}\label{pro22}
For any integers $m$ and $n$ with $2 \leq m \leq n$ and $n \geq 6$, $G_n$ is $(1,m)$-feasible.
\end{pro}
\begin{proof}
We first prove that the assertion is true for $n\leq 10$.
By Lemma \ref{la21}, it holds for $2\leq m \leq 6$ and $n\leq 10$. So, it is sufficient to consider $m \geq 7$. By direct construction, we have the following Hamilton paths for $n=7,8,9,10$ and $m=7,8,9,10$, respectively.
\begin{center}
\begin{tabular}{c|c|c}
      $n$ & $m$ & path \\
    \hline $7$ & $7$ & $\langle 1,3,6,4,2,5,7\rangle$ \\
    \hline $8$ & $7$ & $\langle 1,8,6,3,5,2,4,7\rangle$ \\
    \hline $8$  & $8$ & $\langle 1,3,5,7,2,4,6,8\rangle$\\
    \hline $9$ & $7$ & $\langle 1,3,5,8,6,9,4,2,7\rangle$\\
    \hline $9$ & $8$ & $\langle 1,3,5,7,9,2,4,6,8\rangle$ \\
    \hline $9$ & $9$ & $\langle 1,3,5,8,6,4,7,2,9\rangle$ \\
    \hline $10$ & $7$ & $\langle 1,4,2,9,6,3,5,8,10,7\rangle$\\
    \hline $10$ & $8$ & $\langle 1,4,2,9,7,10,5,3,6,8\rangle$ \\
    \hline $10$ & $9$ & $\langle 1,8,3,5,10,7,2,4,6,9\rangle$ \\
    \hline $10$ & $10$ & $\langle 1,3,5,7,9,2,4,6,8,10\rangle$\\
\end{tabular}
\end{center}

The above cases will serve as the initialization of the following construction. Let $n=5q_1+r_1\geq 11$ and $m=5q_2+r_2$ where $q_1,q_2\in \mathbb{N}, r_1\in [6,10]$ and $r_2\in [2,6]$. Note that $q_2-q_1\leq 1$ since $m\leq n$.

If $q_2\leq q_1$, then by repeatedly applying Lemma \ref{la21} we have a Hamilton path $\langle 1\xrightarrow{[1,6]} 6 \xrightarrow{[6,11]} 11\xrightarrow{[11,16]} 16 \rightarrow\cdots\rightarrow 5q_2+1\xrightarrow{[5q_2+1,n]} m\rangle$. Notice that the existence of the last path $\langle 5q_2+1\xrightarrow{[5q_2+1,n]} m\rangle$ follows from  the existence of a Hamilton path $\langle 1\xrightarrow{[1,n-5q_2]} r_2\rangle$, which can be derived from the fact $n-5q_2\geq 6$ (when $q_2\leq q_1$) and Lemma \ref{la21}.

If $q_2>q_1$, then $q_2=q_1+1$ and similarly we have a Hamilton path $\langle 1\xrightarrow{[1,6]} 6\xrightarrow{[6,11]} 11\xrightarrow{[11,16]} 16 \rightarrow\cdots\rightarrow 5q_1+1\xrightarrow{[5q_1+1,n]} m\rangle$. The last path $\langle 5q_1+1\xrightarrow{[5q_1+1,n]} m\rangle$ exists because we can obtain a Hamilton path $\langle 1\xrightarrow{[1,r_1]} m-5q_1\rangle$ by Lemma \ref{la21} when $r_1=6$ and by the above initialization when $r_1=7,8,9,10.$ Notice that $m-5q_1\leq r_1$ since $m\leq n$.
\end{proof}

%Now we are ready for the construction of a Hamilton path from $1$ to $m$ in $G_n$ for each $m \geq 7$, and $n \geq 11$. First, if $n-m \geq 5$, then by Lemma \ref{la21}, there exists a Hamilton path from $1$ to $2$ in $G_{n-m+1}$ and thus a Hamilton path from $m$ to $m+1$ in $G_{m,n}$. This implies that in $G_{m,n}$, there exists a Hamilton path $Q$ from $m+1$ to $m$.
%On the other hand, since $m\geq 7$, there exists a Hamilton path $P$ from $1$ to $m-1$ by way of recursion. As a consequences, a Hamilton path $\langle P,Q\rangle$ from $1$ to $m$ can be obtained. Now if $n-m<5$, then we find a Hamilton path from $1$ to $n-5$ gives the desired path.\qed

\noindent {\bf Fact}: When $n=5$, $\langle 1,3,5,2,4\rangle$ and $\langle 1,4,2,5,3\rangle$ are Hamilton paths in $G_5$.

Note that, together with the above Fact, Proposition \ref{pro22} can be applied to show that for each $n\geq 5$, $G_n$ is Hamiltonian by simply choosing an integer $m$ such that $m-1$ is a prime; for example, taking $m=4$, a Hamilton path from $1$ to $4$ in $G_n$ can be extended to a Hamilton cycle by adding the edge $\{1,4\}$. We record this as the following corollary.

\begin{cor}[Chen and Sun \cite{oeis}]\label{CS}
The prime difference graph $G_n$ has a Hamilton cycle for any $n\geq 5$.
\end{cor}

In fact, Proposition \ref{pro22} can be further extended to a more general result with arbitrarily designated end-vertices for large $n$. As for small $n$, there does not exist such a path for some particular end-vertices $(n_1,n_2)$. For the sake of completeness, we shall also point out those particular cases in the following lemma.

\begin{la}\label{Hpath}
Let $1\leq n_1 < n_2 \leq n$ and $5\leq n\leq 8$. Then $G_n$ is $(n_1,n_2)$-feasible except for the following cases:\\
when $n=5$, $(n_1,n_2)\in \{(1,2),(2,3),(3,4),(4,5),(1,5)\}$;\\
when $n=6$, $(n_1,n_2)\in \{(2,3),(3,4),(4,5)\}$;\\
when $n=7$, $(n_1,n_2)\in \{(3,4),(4,5)\}$;\\
when $n=8$, $(n_1,n_2)\in \{(4,5)\}$.
\end{la}
\begin{proof} By Proposition \ref{pro22} and the complement property, for $n\geq 6$ it suffices to consider cases of $n_1, n_2\not\in\{1,n\}$. The proof is done by simply listing a corresponding solution for each designated $(n_1,n_2)$, as in Table \ref{table1}.
\end{proof}
\begin{table}[h!]
\begin{center}
    \begin{tabular}{ c| ccccccccccccc}
    $n$ & $(n_1,n_2)$ & path & complement & $(n_1,n_2)$  & path   \\ \hline
    5  & $(1,3)$ & $\langle 1,4,2,5,3\rangle$ & $\leftrightarrow $ & $(3,5)$ & $\langle 5,2,4,1,3\rangle$   \\
       & $(1,4)$ & $\langle 1,3,5,2,4\rangle$ & $\leftrightarrow $ & $(2,5)$ & $\langle 5,3,1,4,2\rangle$   \\
       & $(2,4)$ & $\langle 2,5,3,1,4\rangle$ &  &  &    \\ \hline
    6  & $(2,4)$ & $\langle 2,5,3,6,1,4\rangle$ & $\leftrightarrow $ & $(3,5)$ & $\langle 5,2,4,1,6,3\rangle$   \\
       & $(2,5)$ & $\langle 2,4,6,1,3,5\rangle$ &  &  &    \\ \hline
    7  & $(2,3)$ & $\langle 2,5,7,4,1,6,3\rangle$ & $\leftrightarrow $ & $(5,6)$ & $\langle 6,3,1,4,7,2,5\rangle$   \\
       & $(2,4)$ & $\langle 2,7,5,3,6,1,4\rangle$ & $\leftrightarrow $ & $(4,6)$ & $\langle 6,1,3,5,2,7,4\rangle$   \\
       & $(2,5)$ & $\langle 2,7,4,6,1,3,5\rangle$ & $\leftrightarrow $ & $(3,6)$ & $\langle 6,1,4,2,7,5,3\rangle$   \\
       & $(2,6)$ & $\langle 2,4,7,5,3,1,6\rangle$ &  & &   \\
       & $(3,5)$ & $\langle 3,1,6,4,2,7,5\rangle$ &  & &   \\    \hline
    8  & $(2,3)$ & $\langle 2,5,7,4,1,6,8,3\rangle$ & $\leftrightarrow $ & $(6,7)$ & $\langle 7,4,2,5,8,3,1,6\rangle$   \\
       & $(2,4)$ & $\langle 2,7,5,3,8,6,1,4\rangle$ & $\leftrightarrow $ & $(5,7)$ & $\langle 7,2,4,6,1,3,8,5\rangle$   \\
       & $(2,5)$ & $\langle 2,7,4,6,1,8,3,5\rangle$ & $\leftrightarrow $ & $(4,7)$ & $\langle 7,2,5,3,8,1,6,4\rangle$   \\
       & $(2,6)$ & $\langle 2,4,7,5,3,1,8,6\rangle$ & $\leftrightarrow $ & $(3,7)$ & $\langle 7,5,2,4,6,8,1,3\rangle$   \\
       & $(2,7)$ & $\langle 2,4,1,3,6,8,5,7\rangle$ &  & &   \\
       & $(3,4)$ & $\langle 3,6,1,8,5,2,7,4\rangle$ & $\leftrightarrow $ & $(5,6)$ & $\langle 6,3,8,1,4,7,2,5\rangle$   \\
       & $(3,5)$ & $\langle 3,1,8,6,4,2,7,5\rangle$ & $\leftrightarrow $ & $(4,6)$ & $\langle 6,8,1,3,5,7,2,4\rangle$   \\
       & $(3,6)$ & $\langle 3,1,4,2,7,5,8,6\rangle$ &  & &   \\
    \end{tabular}\\
    \caption{The paths corresponding to the designated $(n_1,n_2)$}\label{table1}
\end{center}
\end{table}

\begin{thm}\label{thm1}
For any integers $n, n_1$ and $n_2$ satisfying $1\leq n_1 <n_2 \leq n$ and $n\geq 9$, $G_n$ is $(n_1,n_2)$-feasible.
\end{thm}
\proof By Proposition \ref{sym}(1), without loss of generality, we may assume that $n_1\leq n-n_2+1$.
Observe that if there exists a Hamilton path $P_1$ from $n_1$ to $n_1 -1$ in $G_{n_1}$ and a Hamilton path $P_2$ from $n_1 +1$ to $n_2$ in $G_{n_{1}+1,n}$, then $\langle P_1,P_2\rangle$ is an $n_1$-to-$n_2$ Hamilton path in $G_n$ unless $n_2=n_1+1$. In the case of $n_2=n_1+1$, we have the same conclusion by slightly modifying $P_2$ to be a Hamilton path from $n_1+2$ to $n_2(=n_1+1)$ in $G_{n_{1}+1,n}$. By Proposition \ref{pro22}, both of the above paths can be obtained as long as $n_1 \geq 6$ since $n-n_1\geq n-n_2+1\geq n_1\geq 6$. Notice that the existence of $P_1$ follows directly from Proposition \ref{sym} and Case 1 of Lemma \ref{la21}.

To complete the proof, it suffices to consider the case of $n_1 \leq 5$.
First, assume $n_2 \geq 7$. In this case, $n\geq 11$ because of the assumption $n_1\leq n-n_2+1$. By symmetry and Proposition \ref{pro22}, there is a Hamilton path from  $n_1$ to 6 in $G_6$ and a Hamilton path from $6$ to $n_2$ in $G_{6,n}$. Concatenating them together yields the desired $n_1$-to-$n_2$ Hamilton path.
Next, assume $n_2\leq 6$. Clearly, we only need to consider $n_1 \geq 2$, and thus there are the following cases:
\[(n_1,n_2)\in \{(2,3),(3,4),(4,5),(5,6),(2,4),(3,5),(4,6),(2,5),(3,6),(2,6)\}.\]
%For convenience, we shall use $a\xrightarrow{S} b$ to denote a Hamilton path from $a$ to $b$ defined on the set $S$ without specifying the internal vertices of this path. For example, if there is a Hamilton path from $5$ to $9$ in $G_{5,10}$, say $\langle 5,7,10,8,6,9\rangle$, then sometimes we write $5\xrightarrow{[5,10]} 9$ to denote the Hamilton path.

By applying Propositions \ref{sym}, \ref{pro22} and Fact, each of the $(n_1,n_2)$ cases mentioned above can be settled by finding a suitable Hamilton path in the following.
%$(2,3): \langle2,4\xrightarrow{[4,n]} 6,1,3\rangle$ for $n\geq 9$\\
%$(3,4): \langle3,1,8\xrightarrow{[5,n]} 5,2,4\rangle$ for $n\geq 9$\\
%$(4,5): \langle4,1,3,8,6,9,7,2,5\rangle$ for $n=9$\\
%$(4,5): \langle4,1,3,6\xrightarrow{[6,n]} 9,2,5\rangle$ for $n\geq 10$\\
%$(5,6): \langle5,2,4,1,3,8 \xrightarrow{[6,n]} 6\rangle$ for $n\geq 10$\\
%$(2,4): \langle2,5\xrightarrow{[5,n]} 8,3,1,4\rangle$ for $n\geq 9$\\
%$(3,5): \langle3,1,4,2,7\xrightarrow{[5,n]} 5\rangle$ for $n\geq 9$\\
%$(4,6): \langle4,1,3,8,5,2,7,9,6\rangle$ for $n=9$\\
%$(4,6): \langle4,1,3,5,2,9\xrightarrow{[6,n]} 6\rangle$ for $n\geq 10$\\
%$(2,5): \langle2,4,1,3,8\xrightarrow{[5,n]} 5\rangle$ for $n\geq 9$\\
%$(3,6): \langle3,1,4,2,9,7,5,8,6\rangle$ for $n=9$\\
%$(3,6): \langle3,1,4,2,5\xrightarrow{[5,n]} 6\rangle$ for $n\geq 10$\\
%$(2,6): \langle2,4,1,3,8,5,7,9,6\rangle$ for $n=9$\\
%$(2,6): \langle2,4,1,3,5,8\xrightarrow{[6,n]} 6\rangle$ for $n\geq 10$.\\
\begin{center}
\begin{tabular}{c|c|c}
           case & $n$ & path \\
    \hline $(2,3)$ & $n\geq 9$ & $\langle2,4\xrightarrow{[4,n]} 6,1,3\rangle$ \\
    \hline $(3,4)$ & $n\geq 9$ & $\langle3,1,8\xrightarrow{[5,n]} 5,2,4\rangle$ \\
    \hline $(4,5)$ & $n=9$     & $\langle4,1,3,8,6,9,7,2,5\rangle$\\
    %\hline $(4,5)$
                   & $n\geq10$ & $\langle4,1,3,6\xrightarrow{[6,n]} 9,2,5\rangle$\\
    \hline $(5,6)$ & $n\geq10$ & $\langle5,2,4,1,3,8 \xrightarrow{[6,n]} 6\rangle$ \\
    \hline $(2,4)$ & $n\geq 9$ & $\langle2,5\xrightarrow{[5,n]} 8,3,1,4\rangle$ \\
    \hline $(3,5)$ & $n\geq 9$ & $\langle3,1,4,2,7\xrightarrow{[5,n]} 5\rangle$\\
    \hline $(4,6)$ & $n=9$     & $\langle4,1,3,8,5,2,7,9,6\rangle$ \\
    %\hline $(4,6)$
                   & $n\geq10$ & $\langle4,1,3,5,2,9\xrightarrow{[6,n]} 6\rangle$ \\
    \hline $(2,5)$ & $n\geq 9$ & $\langle2,4,1,3,8\xrightarrow{[5,n]} 5\rangle$\\
    \hline $(3,6)$ & $n=9$     & $\langle3,1,4,2,9,7,5,8,6\rangle$\\
    %\hline $(3,6)$
                   & $n\geq10$ & $\langle3,1,4,2,5\xrightarrow{[5,n]} 6\rangle$\\
    \hline $(2,6)$ & $n=9$     & $\langle2,4,1,3,8,5,7,9,6\rangle$\\
    %\hline $(2,6)$
                   & $n\geq10$ & $\langle2,4,1,3,5,8\xrightarrow{[6,n]} 6\rangle$\\
\end{tabular}
\end{center}
Notice that for the case $(n_1,n_2)=(5,6)$ we need not consider the case $n=9$ due to the assumption $n_1\leq n-n_2+1$. This completes the proof.\qed

The above two results fully characterize the existence of Hamilton paths with designated end-vertices in the prime difference graphs. As a consequence, we have the following theorem which demonstrates stronger Hamiltonicity of the prime difference graphs and thus improves Chen and Sun's result.

\begin{thm}
    Given $n\geq 5$, for any edge $e$ in the prime difference graph $G_n$ there exists a Hamilton cycle containing $e$.
\end{thm}
\begin{proof}
    Let $n_1$ and  $n_2$ be the end-vertices of the given edge $e$, and without loss of generality assume $1\leq n_1<n_2\leq n$. The existence of a Hamilton cycle passing through $e$ follows directly from the existence of a Hamilton path from $n_1$ to $n_2$ in Lemma \ref{Hpath} and Theorem \ref{thm1}. All we need to concern is whether the forbidden cases described in Lemma \ref{Hpath} are in conflict with the existence of such Hamilton cycles containing $e$. However, the $(n_1,n_2)$ associated with $e$ must be different from that of the forbidden cases since each pair of numbers in the forbidden cases cannot form an edge. This completes the proof.
\end{proof}

\section{2-factors of the prime difference graph}

In this section, we turn our attention to another extension of Hamiltonicity of the prime difference graphs. A {\it 2-factor} is a spanning subgraph of a graph $G$ in which all vertices have degree two, i.e., it is a collection of cycles that together touch each vertex exactly once. Formally, suppose that $H_1, H_2, \cdots, H_k$ are vertex-disjoint subgraphs of $G$ such that $V(G) = \bigcup_{i=1}^kV(H_i)$
and $H_i$ is a cycle for all $i$, $1\leq i\leq k$. Then the union of these $H_i$ is called a 2-factor of
$G$. Obviously, a Hamilton cycle is only one form of many 2-factors of a graph.

Our aim is to show that the prime difference graph $G_n$ contains more about the cycle structure than the existence of a Hamilton cycle. Actually, a simple consequence of Corollary \ref{CS} implies the following result in this direction.

\begin{thm}\label{factor}
For any positive integers $n_i$'s satisfying $n_i\geq 5, n\geq 5$, and $\sum_{i=1}^kn_i=n$, $G_n$ admits  a $2$-factor $\{H_1, H_2, \cdots, H_k\}$ with $|H_i|=n_i$ for all $i$, $1\leq i\leq k$.
\end{thm}
\begin{proof}
    Given a 2-factor containing only cycles of length at least $5$ as described, one can partition $[1,n]$ into consecutive intervals of integers based on these lengths.  Each cycle of the 2-factor can then be obtained by repeatedly applying the shift operation in Proposition \ref{sym} and Corollary \ref{CS}.
\end{proof}

The next result strengthens Theorem \ref{factor} by filling a gap in the cycle lengths $3$ and $4$ of $H_i$'s.

%So, for convenience, we use $3^{\alpha_{1}},4^{\alpha_{2}}\dots n^{\alpha_{n-2}}$ to denote the type of a $2$-factor of $G$. Clearly, $\alpha_{i} \geq 0$ for each $i=1,2,\dots , n-2$ and $\sum ^{n-2}_{i=1} \alpha_{i}(i+2)=n$. Observe that $\alpha_{i}\leq \lfloor \frac{n}{i+2}\rfloor$. Now, we are ready to prove the following result.
\begin{thm}
For positive integer $n\geq 7$, the prime difference graph $G_n$ contains any 2-factor of the complete graph of order $n$ as a subgraph.
\end{thm}
\begin{proof} The proof will be done by recursive constructions. Let $C_\ell$ denote the cycle of length $\ell$ and use $(a_1,a_2,\cdots,a_\ell)$ as a cycle notation to record the cyclic ordering in the cycle. Consider the initial case $n=7$. A $2$-factor in $K_7$ can be either $C_3 \cup C_4$ or $C_7$. The later case is a direct consequence of Theorem \ref{factor} and $(1,3,6)\cup (2,5,7,4)$ is $C_3 \cup C_4$ for the former case.

 From now on, we may assume that $n\geq 8$ and the given 2-factor contains components $C_3$'s or $C_4$'s. The rest of the proof is divided into the following cases based on the length and the number of the minimum cycles in the given 2-factor.

\begin{itemize}
  \item The minimum cycle is $C_4$.
  \begin{itemize}
    \item Only one $C_4$:\\
    $C_4\cup C_5:(1,3,8,6)\cup(2,5,7,9,4)$\\
    $C_4\cup C_6:(1,3,8,6)\cup(2,5,10,7,9,4)$\\
    $C_4\cup C_7:(1,3,8,6)\cup(2,5,10,7,9,11,4)$\\
    $C_4\cup C_8:(1,3,8,6)\cup(2,5,10,7,12,9,11,4)$\\
    $C_4\cup C_{n-4} (n\geq 13):(2,5,7,4)\cup(6,1,3,8\xrightarrow{[8,n]} 9)$ by Propositions \ref{Hpath} and \ref{sym}(2).\\
    The remaining case is the 2-factor consisting of $\{C_4, H_1, H_2, \cdots\}$ where $H_i$'s are cycles of length at least 5. We first construct $C_4\cup H_1$ on $G_{4+|H_1|}$ according to the previous construction and then reduce the given 2-factor on $G_n$ to a smaller 2-factor $\{H_2, \cdots\}$ of cycle length at least 5 on $G_{5+|H_1|,n}$. Obviously, by the previous discussion, the desired 2-factor can be obtained.

    \item More than one $C_4$:\\
    $2C_4:(1,3,8,6)\cup(2,5,7,4)$\\
    $3C_4:(1,3,10,8)\cup(2,5,12,7)\cup(9,11,6,4)$\\
    For the remaining cases, the given 2-factor either consists of only $C_4$'s or contains at least one $C_i$ with $i\geq 5$. The later case can be done by recursively constructing $2C_4$'s and reduce it down to the cases of only one $C_4$ or all cycles of length at least 5;  while the former case can be done by simply constructing $2C_4$'s (or sometimes $3C_4$ if necessary) repeatedly.
  \end{itemize}
  \item The minimum cycle is $C_3$.
  \begin{itemize}
    \item Only one $C_3$:\\
    $C_3\cup C_5:(2,5,7)\cup(1,3,8,6,4)$\\
    $C_3\cup C_6:(1,3,8)\cup(2,5,7,9,6,4)$\\
    $C_3\cup C_7:(1,3,8)\cup(2,5,10,7,9,6,4)$\\
    $C_3\cup C_8:(1,3,8)\cup(2,5,10,7,9,11,6,4)$\\
    $C_3\cup C_{n-3} (n\geq 12):(1,3,6)\cup(5,2,4,7\xrightarrow{[7,n]} 8)$ by Lemma \ref{Hpath} and the shift operation in Proposition \ref{sym}.\\
    $C_3\cup 2C_4:(1,3,8)\cup(4,6,9,11)\cup(2,5,10,7)$\\
    For the remaining cases, which are of the type $\{C_3, H_1, H_2, \cdots\}$ with $H_i$'s being cycles of length at least 4, we first construct $C_3\cup H_1$ on $G_{3+|H_1|}$ according to the previous construction and then reduce the given 2-factor on $G_n$ to a smaller 2-factor $\{H_2, \cdots\}$ of cycle length at least 4 on $G_{4+|H_1|,n}$. Obviously, by the previous discussion, the desired 2-factor can be obtained.

    \item Exactly two $C_3$:\\
    $2C_3\cup C_4:(1,3,8)\cup(4,6,9)\cup(2,5,10,7)$\\
    $2C_3\cup C_5:(1,3,8)\cup(5,10,7)\cup(2,9,11,6,4)$\\
    $2C_3\cup C_{n-6} (n\geq 12):(1,3,6)\cup(2,4,7)\cup (5,8\xrightarrow{[8,n]} 10)$ by Lemma \ref{Hpath} and the shift operation in Proposition \ref{sym}.\\
    For the remaining cases, the given 2-factor contains at least four cycles. We first construct a $C_3\cup C_i$ for some $i\geq 4$ on $G_{i+3}$ and then this case is reduced to the previous case (only one $C_3$) on $G_{i+4,n}$.

    \item More than two $C_3$:\\
    $3C_3:(1,3,8)\cup(2,5,7)\cup(4,6,9)$\\
    $4C_3:(1,3,8)\cup(2,7,9)\cup(4,6,11)\cup(5,10,12)$\\
    $3C_3\cup C_4:(1,3,8)\cup(2,7,9)\cup(5,10,12)\cup(4,6,13,11)$\\
    If the given 2-factor consists of four cycles, then it can be of types $4C_3, 3C_3\cup C_4$, or $3C_3\cup C_i$ for some $i\geq 5$. The former two cases can be constructed as the above. The last case can be obtained by first constructing $3C_3$ on $G_9$ and then the cycle $C_i$, $i\geq 5$, by Corollary \ref{CS}.\\
    $5C_3:(1,3,6)\cup(2,4,15)\cup(5,7,10)\cup(8,11,13)\cup(9,12,14)$\\
    For the rest, it suffices to consider the case that the 2-factor contains at least five cycles, which can be one of the following cases.\\
    $mC_3 (m\geq 6):$ This can be obtained immediately by using combinations of $3C_3$'s and $4C_3$'s.  \\
    $mC_3\cup C_4 (m\geq 4):$ We first construct $C_3\cup C_4$ on $G_7$ and then this case is reduced to the case $(m-1)C_3$, which can be  obtained by the above discussion.\\
    $mC_3\cup H_1\cup H_2\cdots \cup H_k (m\geq 3, k\geq 2):$ We first construct $C_3\cup H_1$ on $G_{3+|H_1|}$ and this case is reduced to the case that contains at least two $C_3$ and one $H_i$ of cycle length at least 4. The rest can be done by the previous discussion. This completes the proof.
  \end{itemize}
\end{itemize}
\end{proof}

Remark that the condition $n\geq 7$ in the above theorem is tight as there does not exist two vertex-disjoint $C_3$ in $G_6$.

\section{Hamilton cycles generated by two primes}

We start with an interesting observation: does $G_n$ contain a Hamilton cycle whose edges are all obtained by differences $2$ or $3$. For example, there exists such a Hamilton cycle $(1,4,2,5,3)$ in $G_5$. However, it does not for some particular cases, especially when $n$ is small. In fact, we find that no such Hamilton cycles exist for $n\in [1,9]\setminus\{5\}$. On the other hand, we prove that for each $n\geq10$, such a Hamilton cycle does exist.

\begin{la}\label{Hpath23}
For each $n\geq6$, there exists a Hamilton path such that all edges are of differences either $2$ or $3$.
\end{la}
\begin{proof}
    If $n$ is even, then $\langle n, n-2, \dots,8,6,3,1,4,2,5,7,\dots,n-3,n-1\rangle$ is the desired path. Otherwise, $n$ is odd, $\langle n, n-2, \dots,7,5,2,4,1,3,6,8,\dots,n-3,n-1\rangle$ is the desired path.
\end{proof}

\begin{pro}
For each $n\geq10$, $G_n$ contains a Hamilton cycle whose edges are of prime differences $2$ or $3$.
\end{pro}
\begin{proof}
    Since $n\geq10$, both the sets $A=\{1,2,\dots,\lfloor\frac n2\rfloor+1\}$ and $B=\{\lfloor\frac n2\rfloor,\lfloor\frac n2\rfloor+1,\dots,n\}$ are of cardinality at least $6$. Therefore, by Lemma \ref{Hpath23} and the complement property of Proposition \ref{sym}, we have a Hamilton path $\langle \lfloor\frac n2\rfloor+1 \xrightarrow{[1,\lfloor\frac n2\rfloor+1]} \lfloor\frac n2\rfloor\rangle$ defined on $A$ and a Hamilton path $\langle \lfloor\frac n2\rfloor \xrightarrow{[\lfloor\frac n2\rfloor,n]} \lfloor\frac n2\rfloor+1\rangle$ defined on $B$ using differences $2$ and $3$. This implies a Hamilton cycle in $G_n$ using only differences $2$ and $3$.
\end{proof}

Next, we consider a more general situation in using two prime differences.
Adapting the idea of abstract algebra, it is not difficult to see that $G_n$ contains a Hamilton cycle generated by the two primes $p$ and $q$ if $p+q=n$. For example, the Hamilton cycle $(1,3,5,7,9,2,4,6,8)$ in $G_9$ uses $2$ and $7$ as prime differences. Therefore, we have the following proposition and its corollaries.

\begin{pro}
Let $p$ and $q$ be two distinct primes such the $p+q=n$. Then, $G_n$ contains a Hamilton cycle with only prime differences $p$ and $q$.
\end{pro}

\begin{proof}
Notice that $p$ and $n$ are relatively prime, as well as $q$ and $n$. A fundamental result in abstract algebra asserts that $p$ and $q$ are generators of the group $(\mathbb{Z}_n,+)$ with the ordinary addition as its operation. This implies the subgroup generated by $p$ is cyclic. %$\mathbb{Z}_n=\langle p\rangle=\langle q\rangle$ is cyclic.
Therefore, one can conclude that the cyclic ordering $(p,2p,3p,\dots)$ generated by $p$ corresponds to the vertex ordering of the desired Hamilton cycle in $G_n$ since the difference of any two adjacent vertices, congruent to $p$ modulo $n$, is either $p$ or $q$. %But in $G_n$, if $v>u$ then $v-u=p$ trivially. If $u>v$ then $u-v=-p=q$. Then we complete the proof. Note that $\langle q\rangle$ and $\langle p\rangle$ is the same cycle in reversing order.
\end{proof}

\begin{cor}
If $(p, p+2)$ is a twin prime pair, then $G_{p+2}$ contains a Hamilton cycle generated by the differences $2$ and $p$.
\end{cor}

Since the edges with distinct differences in $G_n$ are different edges, the following result is clear.

\begin{cor}\label{2primes}
If $n=p_1+q_1=p_2+q_2$ where $p_1,p_2,q_1$ and $q_2$ are distinct primes, then $G_n$ contains two edge-disjoint Hamilton cycles.
\end{cor}

The above discussion provides clues about existence of many edge-disjoint Hamilton cycles in the prime difference graphs. For example, $20=3+17=7+13$, so $G_{20}$ contains two edge-disjoint Hamilton cycles generated by $3$ and $17$, and $7$ and $13$, respectively. Here is another example; $30=7+23=11+19=13+17$ implies that $G_{30}$ contains at least three edge-disjoint Hamilton cycles generated by the corresponding prime pairs. With the one of differences $2$ and $3$, there are at least four edge-disjoint Hamilton cycles in $G_{30}$.

Recent developments in number theory have provided further evidence on this direction. According to the Green–Tao theorem, there exist arbitrarily long sequences of primes in arithmetic progression.
\begin{thm}[Green and Tao \cite{GreenTao2008}]\label{gt} For any given $k\geq1$, there exists an arithmetic progression of primes $p_1,p_2,\dots,p_k$.
\end{thm}
This, together with Corollary \ref{2primes}, implies the following result.
\begin{cor}
For any $t\in\mathbb{N}$, there exists a number $n(t)$ such that $G_{n(t)}$ contains $t$ edge-disjoint Hamilton cycles.
\end{cor}
\begin{proof}
Given $t\in\mathbb{N}$, let $k=2t$ and $n=p_i+p_{2t-i+1}$ for $i=1,2,\dots,t$, where $p_i$'s are primes obtained as in Theorem \ref{gt}. By Corollary \ref{2primes}, $G_n$ contains at least $t$ edge-disjoint Hamilton cycles.\end{proof}

\noindent{\bf Remark.} Among previous results, several constructions of Hamilton cycles in the prime difference graphs are provided. We find that only a few number of primes are sufficient to construct a single Hamilton cycle. According to the prime number theorem, the asymptotic distribution of the prime numbers among the positive integers in $[n]$ is close to $\frac{n}{\log n}$, which implies the degree of vertices in $G_n$ is $O(\frac{n}{\log n})$. This means that it is likely to find many edge-disjoint Hamilton cycles if one can exploit the primes properly to form  Hamilton cycles. We believe that the number of edge-disjoint Hamilton cycles in $G_n$ is getting larger as $n$ tends to infinity, and wonder if this number could be asymptotically close to $\frac{n}{\log n}$. It would be very interesting if one can provide an asymptotic bound on the number of edge-disjoint Hamilton cycles in $G_n$.

\section*{Acknowledgments}
The authors would like to thank Professor Z. W. Sun for valuable suggestions.

\end{document}